\documentclass[12pt,a4paper]{amsart}
\usepackage{amscd}
\usepackage{amssymb,amsfonts,latexsym}
\usepackage{fixltx2e}
\usepackage{mathscinet}

\newtheorem{thm}{Theorem}[section]
\newtheorem{cor}[thm]{Corollary}
\newtheorem{lem}[thm]{Lemma}
\newtheorem{prop}[thm]{Proposition}
\theoremstyle{definition}

\theoremstyle{remark}
\newtheorem{rem}[thm]{Remark}
\newtheorem*{rem*}{Remark}
\newtheorem{ex}[thm]{Example}
\newtheorem*{claim}{Claim}
\numberwithin{equation}{section}

\newcommand{\Bm}[1]{\textup{(B$_{#1}$)}}
\newcommand{\mQ}{\mathbb{Q}}
\newcommand{\ra}{\rightarrow}
\newcommand{\fp}{\frak{p}}
\newcommand{\fq}{\frak{q}}
\newcommand\SBr{{\operatorname{SBr}}}
\def\divides{{\,|\,}}
\def\ndivides{{\not|\,}}
\newcommand{\set}[1]{\{#1\}}
\newcommand{\gen}[1]{\l<#1\r>}
\newcommand{\field}[1]{\mathbb{#1}}
\newcommand{\Q}{\field{Q}}
\newcommand{\N}{\field{N}}
\newcommand{\Z}{\field{Z}}
\newcommand{\F}{\field{F}}
\newcommand{\Ps}{\field{P}}
\newcommand{\pp}{{\mathfrak p}}
\newcommand{\PP}{{\mathfrak P}}
\newcommand{\sett}[2]{\{#1\,|\,#2\}}

\DeclareMathOperator{\l<}{\langle}
\DeclareMathOperator{\r>}{\rangle}
\DeclareMathOperator{\Gal}{Gal}
\DeclareMathOperator{\Br}{Br}
\DeclareMathOperator{\ind}{ind}
\DeclareMathOperator{\Hom}{Hom}
\DeclareMathOperator{\inv}{inv}
\DeclareMathOperator{\charak}{char} 
\DeclareMathOperator{\ord}{ord}

\begin{document}

\title[Noncrossed products over Henselian fields]{Noncrossed product bounds over Henselian fields}

\begin{abstract} The existence of finite dimensional central division algebras with no maximal subfield that is Galois over the center (called noncrossed products), was for a time the biggest open problem  in the theory of division algebras, 
before it was settled by Amitsur.

Motivated by Brussel's discovery of noncrossed products over $\Q((t))$,
we describe the ``location'' of noncrossed products in the Brauer group of general Henselian valued fields with arbitrary value group and global residue field.
We show that within the fibers defined canonically by Witt's decomposition of the Brauer group of such fields, 
crossed products and noncrossed products are, roughly speaking, separated by an index bound.
This generalizes a result of the first and third author for rank $1$ valued Henselian fields.

Furthermore, we prove that all fibers which are not covered by the rank $1$ case,
and where the characteristic of the residue field does not interfere,
contain noncrossed products.

We show by example that, unlike in the rank 1 case, the value of the
index bound does not depend on the number of roots of unity that are
present.
Thus, the index bounds are in general of a different nature than in the rank 1 case.
\end{abstract}

\def\Tech{Department of Mathematics, Technion --- Israel Institute of Technology, Haifa 32000, Israel}

\author{Timo Hanke}
\address{
Lehrstuhl D f\"ur Mathematik\\
RWTH Aachen\\
Templergraben 64\\
D-52062 Aachen\\
Germany}
\email{hanke@math.rwth-aachen.de}

\author{ Danny Neftin}
\address{
University of Michigan,
530 Church St.,
Ann Arbor, MI 48109-1043
USA}
\email{neftin@umich.edu}

\author{ Jack Sonn }
\address{\Tech}
\email{sonn@math.technion.ac.il}

\date{\today}

\thanks{The first author would like to thank Eli Aljadeff for the financial support of a visit to the Technion in 2010.
The second author would like to thank Julia Hartmann for the financial support of a visit to RWTH Aachen in 2010.}



\maketitle

\section{Introduction}
A finite-dimensional  division algebra over its center  $F$ is called a {\em crossed
  product} if it has a maximal commutative subfield which is Galois over $F$,
otherwise a {\em noncrossed product}.

After Amitsur settled the fundamental long-standing problem of existence of
noncrossed products in \cite{amitsur:central-div-alg},
they were subsequently discovered over more familiar fields.
Most notably, Brussel showed in \cite{brussel:noncr-prod}, \cite{brussel:noncr-prod-ff}
that noncrossed products exist over
complete discrete rank $1$ valued fields with a global residue field\footnote{By a global field we mean a finite extension of $\Q$
  or a finite extension of $\F_q(t)$ where $\F_q$ is a finite field.}, e.g.\ over $\mQ((x))$.
From this basic case,  their existence over many other fields was derived,
e.g.\ over all finitely generated fields that are neither finite nor global
(\cite{brussel:noncr-prod-ff}),
and over all function fields of curves over complete discrete valuation rings
(\cite{brussel:kpt}, \cite{Chen10I}).

The basic setup used for Brussel's discovery is  Witt's description  (\cite{witt:schiefkoerper}) of the Brauer group of a complete discrete rank $1$ valued field. More precisely, Witt's theorem describes the inertially split part $\SBr(F)$  which consists of all elements of the Brauer group that are split by an unramified extension (cf.\ \cite[\S5]{hanke:maxsf}).
Witt's theorem applies more generally to Henselian fields $F$ with arbitrary value group $\Gamma$ 
(see \cite{scharlau:br-henselkoerper}, \cite[\S5]{jacob-wadsworth:div-alg-hensel-field}, \cite[Theorem~3.2]{aljadeff-sonn-wasdworh:projective-schur}), 
e.g.\ to the field of iterated Laurent series $\mQ((x_1))\cdots ((x_r))$, 
and gives an isomorphism:
\begin{equation}\label{eq:seq}
\SBr(F)\cong \Br(K)\oplus \Hom(G_K,\Delta/\Gamma)
\end{equation}
where $K$ is the residue field and $G_K$ is its absolute Galois group,
 $\Delta$ is the divisible hull of $\Gamma$,
 $\Delta/\Gamma$ is equipped with the discrete topology,
and $\Hom$ is the group of continuous homomorphisms.

Assume $K$ is a global field.
In \cite{hanke-sonn:location}, the first and third author
took the approach of fixing an element $\chi\in \Hom(G_K, \Delta/\Gamma)$ and considering the {\em fiber } $\Br(K)+\chi$. We call $\chi$ {\it cyclic} if $\Im(\chi)$ is cyclic.
For cyclic $\chi$, \cite{hanke-sonn:location} shows that for every $N\in\mathbb{N}$ there are only two possible cases:
\begin{enumerate}
\item[(I)] all division algebras of index $N$ in the fiber over $\chi$  are crossed products;
\item[(II)] the fiber over $\chi$ contains infinitely many noncrossed products of index $N$.
\end{enumerate}
Furthermore, and most importantly, there are bounds on the exponents in the
prime decomposition of $N$ such that case (I) occurs ``below the
bounds'' and case (II) ``above''.
A precise formulation of this result is the special case of Theorem 1.1 below in which $\chi$ is assumed to be cyclic\footnote{In order to see that Theorem 1.1 for cyclic $\chi$ was proved in \cite{hanke-sonn:location} not only for rank 1 valued fields, one should replace \cite[Brussel's Lemma,~p.322]{hanke-sonn:location} by \cite[Corollary~5]{hanke:maxsf}.}.

All that was known about the appearance of noncrossed products in the more complicated case of non-cyclic $\chi$
were two examples (\cite{hanke:expl-ex}, \cite{Coyette}) in which noncrossed products of index $8$ appear in the fiber over $\chi$.
In both of them, $\Im(\chi)$ is the Klein 4 group.

In the present paper, we show that the phenomenon from \cite{hanke-sonn:location} holds more generally for arbitrary $\chi$ (Theorem\ \ref{thm:main}),
and that ``away from $\charak K$'' noncrossed products appear in every noncyclic fiber (Theorem~\ref{thm:finitness}).

Note that by \cite[Theorem 5.15(a)]{jacob-wadsworth:div-alg-hensel-field} the index of an element in the fiber over $\chi$ is always a multiple of $|\chi|:=|\Im(\chi)|$.

\begin{thm}\label{thm:main}
There exists a collection of bounds
$b_p=b_{p}(\chi)\in\mathbb{N}\cup\{\infty\}$ where $p$ runs through the
rational primes,
such that for every natural number $m=\prod p^{n_p}$:
\begin{enumerate} \item[(a)] if $n_p\leq b_p(\chi)$ for all $p$, then all
  division algebras of index $m|\chi|$ in the fiber over $\chi$ are crossed products.
\item[(b)] if $n_p>b_p(\chi)$ for some $p$, then the fiber over $\chi$ contains infinitely many noncrossed products of index $m|\chi|$.
\end{enumerate}
\end{thm}
Our proof of Theorem \ref{thm:main} includes the cyclic case and is significantly simpler than \cite{hanke-sonn:location}.

Note that $b_p(\chi)=\infty$ is allowed and hence, as shown in  \cite{hanke-sonn:location},
it may happen that for some cyclic $\chi$ only  (I) occurs.  
However, in striking contrast to the cyclic case, we show:

Let $K(\chi)$ be the fixed field of the kernel of $\chi$. 
\begin{thm}\label{thm:finitness}
If $p\neq\charak K$ and the $p$-Sylow subgroup of $\Im(\chi)$ is non-cyclic then
\begin{enumerate}
\item $b_p(\chi)<\infty$,
\item  $K(\chi)$ does not contain the $p$-th roots of unity implies $b_p(\chi)=0$. 
\end{enumerate}
In particular, the fiber over $\chi$ contains noncrossed products
whenever the maximal prime-to-$\charak K$ subgroup of $\Im(\chi)$ is non-cyclic.
\end{thm}
This is in contrast to the cyclic case because, according to \cite{hanke-sonn:location},
$b_p(\chi)=\infty$ can occur for cyclic $\chi$ even if $K(\chi)$ does not contain the $p$-th roots of unity.
Thus, neither of statements (i) and (ii) of Theorem \ref{thm:finitness} holds for cyclic $\chi$.

In Section \ref{examples.sec}, we demonstrate how a description of the bounds obtained from the proof of Theorem \ref{thm:main} can be used to compute the bounds in examples and to obtain new noncrossed products of low index.

\section{Existence of bounds}\label{sec:exist}

\subsection{Setup}\label{setup.sec}
Let $F$ be a Henselian valued field with value group $\Gamma$ and
residue field a global field $K$.
Let $\Delta$ denote the divisible hull of $\Gamma$.
Let $G_K$ denote the absolute Galois group of $K$.
The group $\Hom(G_K,\Delta/\Gamma)$ is understood as the group of continuous
homomorphisms, $\Delta/\Gamma$ equipped with the discrete
topology.

For any $\chi\in\Hom(G_K,\Delta/\Gamma)$, we write $K(\chi)$ for the
fixed field of the kernel of $\chi$,
which is a finite abelian extension of $K$ with Galois group $\Im(\chi)$.
Throughout the paper, we consider the character $\chi$ as fixed and set $M:=K(\chi)$.

Let $\Ps$ be the set of finite rational primes and for $p\in \Ps$ denote 
by $v_p(n)$ the maximal exponent $e$ such that $p^e|n$ for any natural number~$n$.

\subsection{Theorem 1.1}\label{sec:thm}
For $\alpha\in\Br(K)$ we denote by $\alpha^M$ the image of $\alpha$ in $\Br(M)$ under the restriction map.
The index formula \cite[Theorem.\ 5.15(a)]{jacob-wadsworth:div-alg-hensel-field} gives: \footnote{\cite[Theorem.\ 5.15(a)]{jacob-wadsworth:div-alg-hensel-field} does not require that $K$ is a global field.}
$$ \ind(\alpha+\chi)=|\chi|\ind\alpha^M.$$
Therefore, in order to prove Theorem \ref{thm:main}, we consider the following condition on the field $K$.
\begin{equation}
  \label{eq:Im}
  \tag{I$_m$}
  \text{\parbox[c]{10cm}{For every $\alpha\in\Br(K)$ with $\ind\alpha^M=m$,
the division algebra underlying $\alpha+\chi$ is a crossed product.}}
\end{equation}
For a global field $K$ we shall prove the existence of bounds $b_p(\chi)$ such that \eqref{eq:Im} holds if and only if $v_p(m)\leq b_p(\chi)$ for all $p\in\Ps$.
The details of our proof will reveal that if \eqref{eq:Im} does not hold then there are in fact infinitely many $\alpha\in\Br(K)$ with $\ind\alpha^M=m$ such that the division algebra underlying $\alpha+\chi$ is a noncrossed product (Remark \ref{infinitely.rem}).
Then the proof of Theorem \ref{thm:main} will be completed.

\subsection{Galois Covers}\label{sec:cov}
We say that a field $L\supseteq M$ is a {\em cover} of $M/K$ if $L/K$ is Galois.
In this case, we call $m:=[L:M]$ the {\em degree} of the cover
and speak of $L$ as an {\em $m$-cover}.

The division algebra underlying $\alpha+\chi$ is a crossed product if
and only if the division algebra underlying $\alpha^M$ contains a maximal subfield which is Galois over $K$
(see \cite[Corollary~5]{hanke:maxsf} or \cite[p.\ 381, Corollary]{brussel:noncr-prod} for complete
discrete rank-$1$ valued $F$).
Such maximal subfields are characterized as the $m$-covers of $M/K$
that split $\alpha$, where $m=\ind\alpha^M$
(see e.g.\ \cite[Corollary~13.3]{pierce:ass-alg}).
Condition \eqref{eq:Im} is therefore equivalent to:
\begin{equation}
  \label{eq:Am}
  \tag{A$_m$}
\text{\parbox[c]{10cm}{Every $\alpha\in\Br(K)$ with $\ind\alpha^M=m$ is split by an $m$-cover of $M/K$.}}
\end{equation}
\begin{rem}
(i) For the equivalence of \eqref{eq:Im} and \eqref{eq:Am} it is not required that $K$ is a global field.

(ii)
Condition \eqref{eq:Am} is a condition on $M=K(\chi)$ rather than on $\chi$.
\eqref{eq:Am} can be considered more generally for any Galois extension $M/K$.
In fact, the form of $M$ turns out to be irrelevant for the remaining part of \S\ref{sec:exist},
i.e.\ from now on $M$ may be replaced by an arbitrary finite Galois extension of the global field $K$.
\end{rem}

\subsection{Covers with prescribed local degrees}\label{sec:covglob}
By a prime $\pp$ of $K$ we mean a finite or infinite prime,
and we write $K_\pp$ for the completion at $\pp$.
Let $L$ be a cover of $M/K$.
We call $[L:M]_\pp:=[L_\PP:M_{\PP\cap M}]$ the {\em local degree} of $L$ at $\pp$,
where $\PP$ is any prime of $L$ dividing $\pp$.
Our first goal is to translate \eqref{eq:Am} to a condition about the existence of covers of $M/K$
with prescribed local degrees at finitely many primes of $K$
(see Proposition \ref{prop:fld-reduction} below).

Let $\alpha\in\Br(K)$.
Since $K$ is a global field, by the theorem of Albert-Hasse-Brauer-Noether (see e.g.\ \cite[\S18.4]{pierce:ass-alg}),
the splitting fields of $\alpha$ are characterized by their local degrees, i.e.\ for any Galois extension $L/K$:
\begin{equation}
  \label{eq:globsplit}
 \text{\parbox[c]{11cm}{
$L$ splits $\alpha$ if and only if $\ind_\pp\alpha|[L:K]_\pp$ for all primes $\pp$ of $K$,}}
\end{equation}
where $\ind_\pp\alpha:=\ind\alpha^{K_\pp}$.
Hence, for any cover $L$ of $M/K$:
\begin{equation}
  \label{eq:coversplit}
 \text{\parbox[c]{11.1cm}{
$L$ splits $\alpha$ if and only if $\ind_\pp\alpha^M|[L:M]_\pp$ for all primes $\pp$ of $K$,}}
\end{equation}
where $\ind_\pp\alpha^M:=\ind\alpha^{M_\PP}$ for any prime $\PP$ of $M$ dividing $\pp$.
We analyze in detail the possible local indices $\ind_\pp\alpha^M$ for all $\alpha\in\Br(K)$ with $\ind\alpha^M=m$ for a given $m$.

For a fixed $p\in\Ps$, let $u_1,u_2,\ldots$ be the family of numbers
$v_p([M:K]_\fp)$ where $\fp$ runs over all primes of $K$,
ordered so that $u_1\geq u_2\geq\ldots$.
If $u_1>u_2$ then the unique prime $\fp$ of $K$ with
$v_p([M:K]_\fp)=u_1$ is called {\em $p$-isolated} in $M/K$.
We denote by $g_p$ the gap $u_1-u_2$, so that $g_p>0$ if and only if there is a
$p$-isolated prime. We shall call a prime $\fp$ of $K$ {\it isolated} if it is $p$-isolated for some $p$.

Note that if $M/K$ is cyclic then, as a consequence of the Chebotarev density theorem, there are no isolated primes.
Moreover, any prime that is isolated in a non-cyclic extension $M/K$ must ramify non-trivially.
Hence, only finitely many primes can be isolated in any given $M/K$.
Infinite primes are never isolated.

\begin{lem}\label{isolated.lem}
Let $\alpha\in \Br(K)$.
If $\pp$ is $p$-isolated in $M/K$ then
\begin{equation}\label{divisible.equ}
v_p(\ind_{\fp}\alpha^M)\leq\max\{v_p(\ind\alpha^M)-g_p,0\}.
\end{equation}
\end{lem}

\begin{proof}[Proof of Lemma \ref{isolated.lem}]
  Set $k:=v_p(\ind_{\fp}\alpha^M), n:=v_p(\ind\alpha^M)$, and let $u_1,u_2$ be as above.
  Clearly, $n\geq k$.
The assertion to prove is: $k=0$ or $n\geq k+g_p$.
Assume $k>0$.
Since $v_p([M:K]_{\fp})=u_1$ and $k>0$ we have $v_p(\ind_{\fp}\alpha)=k+u_1$.
Since the sum of Hasse invariants of $\alpha$ is $0$,
there exists a prime $\fp_1\neq\fp$ of $K$ with
$v_p(\ind_{\fp_1}\alpha)\geq k+u_1$.
But $v_p([M:K]_{\fp_1})\leq u_2$, hence
$n\geq v_p(\ind_{\fp_1}\alpha^M)\geq k+u_1-u_2=k+g_p$.
\end{proof}

Let $m\in\N$.
For a finite prime $\fp$ of $K$ define $d_\pp(m)\in\N$ by requiring for every $p\in\Ps$,
$v_p(d_\pp(m))=\max\{v_p(m)-g_p,0\}$ if $\fp$ is $p$-isolated,
and $v_p(d_\pp(m))=v_p(m)$ otherwise.
For a real (resp.\ complex) prime $\fp$ of $M$ set $d_\pp(m):=\gcd(m,2)$ (resp.\ $d_\pp(m):=1$).
Clearly, $d_\pp(m)|m$ for any $\pp$.

\begin{prop}\label{prop:fld-reduction}
Condition \eqref{eq:Am} is equivalent to:
\begin{equation}\label{eq:Bm}
  \tag{B$_m$}
\text{\parbox[c]{10cm}{For every finite set $S$ of primes of $K$,
$M/K$ has an $m$-cover $L$ such that $d_\pp(m)\divides [L:M]_{\fp}$ for any $\pp\in S$.}}
\end{equation}
\end{prop}
\begin{proof}[Proof of Proposition \ref{prop:fld-reduction}]
\eqref{eq:Bm}$\Rightarrow$\eqref{eq:Am}:
Let $\alpha\in \Br(K)$ with $\ind \alpha^M=m$.
We construct an $m$-cover $L$ of $M/K$ that splits $\alpha$.
Let $S$ be the set of primes $\fp$ of $K$ such that $\ind_\fp\alpha\neq 1$.
By assumption, there is an $m$-cover $L$ of $M/K$ such that $d(m)_{\fp}\divides [L:M]_{\fp}$ for each $\fp\in S$.
Let $\fp\in S$ and $p\in\Ps$.
If $\fp$ is not $p$-isolated then
$$v_p(\ind_{\fp}\alpha^M)\leq v_p(m)=v_p(d_\pp(m))\leq v_p([L:M]_{\fp}).$$
If $\fp$ is $p$-isolated then Lemma \ref{isolated.lem} implies
$$v_p(\ind_{\fp}\alpha^M)\leq \max\{v_p(m)-g_p,0\}=v_p(d_\pp(m))\leq v_p([L:M]_{\fp}).$$
Thus, $\ind_\pp\alpha^M|[L:M]_\pp$ for any $\pp$, i.e.\ $L$ splits $\alpha$ by \eqref{eq:coversplit}.

\eqref{eq:Am}$\Rightarrow$\eqref{eq:Bm}:
Let $S$ be any finite set of primes of $K$.
We construct the element $\alpha\in\Br(K)$, to which we apply
\eqref{eq:Am}, by its primary decomposition $\alpha=\prod_{p|m}\alpha_p$.
To this end, fix a $p\in\Ps$ with $p|m$ and let $n:=v_p(m)$.
We explain below how to choose the Hasse invariants of $\alpha_p\in \Br(K)$
such that
$\ind_\fp\alpha_p^M=p^n$ for any non-$p$-isolated $\fp\in S$
and such that
$\ind_\fp(\alpha_p^M)=p^{\max\{n-g_p,0\}}$ for any $p$-isolated $\fp\in S$.
Then $\alpha:=\prod_{p|m}\alpha_p$ clearly satisfies $d_\pp(m)|\ind_\fp\alpha^M$ for any $\fp\in S$.
Since by enlarging $S$ we can assume wlog.\ that $S$ contains non-$p$-isolated primes for each $p|m$,
this implies $\ind\alpha^M=m$.
Applying \eqref{eq:Am} to $\alpha$, there is an $m$-cover $L$ of $M/K$ which splits $\alpha^M$.
By \eqref{eq:coversplit}, such an $L$ satisfies $d(m)_p\divides [L:M]_\fp$ for any $\fp\in S$.

We now define $\alpha_p$ as above by its Hasse invariants, denoted $\inv_\pp\alpha_p$.
Let $u_1,u_2$ be as above.
By enlarging $S$, assume wlog.\ that $S$ contains two primes $\pp_1,\pp_2$ with $v_p([M:K]_{\fp_i})=u_i$.
(Either $\pp_1$ is $p$-isolated or $M/K$ has no $p$-isolated primes.)
Set $\inv_\fp\alpha_p:=0$ for any $\fp\not\in S$.
At every $\fp\in S\setminus\set{\pp_1,\pp_2}$ choose $\inv_\fp\alpha_p$ to be of order $p^r$ where $r=n+v_p([M:K]_\fp)$.
The order of all invariants we have set so far is at most $p^{n+u_2}$.
We can thus choose $\inv_{\pp_2}\alpha_p$ of order $p^{n+u_2}$ such that
$x:=\sum_{\pp\neq\pp_1}\inv_\fp\alpha_p$ has order equal to $p^{n+u_2}$.
(In the case $p^{n+u_2}=2$ we further assume wlog.\ that $|S|$ is even.)
Setting $\inv_{\fp_1}\alpha_p:=-x$ defines an $\alpha_p\in \Br(K)$ with
$\ind_\fp(\alpha_p^M)=p^n$ for every $\fp\in S\setminus\set{\pp_1}$
and with $\ind_{\fp_1} \alpha_p^M=p^{n+u_2-u_1}=p^{n-g_p}$ if $n\geq g_p$ and $1$ otherwise, as desired.
\end{proof}
\begin{rem}\label{infinitely.rem} The proof of Proposition \ref{prop:fld-reduction} shows that if there are noncrossed products $\alpha+\chi$ with $\ind(\alpha^M)=m$ then there are infinitely many such noncrossed products. Indeed, if $(B_m)$ fails for a set $S_0$, it fails for every set $S$ containing $S_0$. The proof reveals that for every such set $S$, there is an $\alpha$ whose Hasse invariants are non-zero at primes of $S$ and $\alpha+\chi$ is a noncrossed product. In particular, there are infinitely many such classes $\alpha$.
\end{rem}

For a prime $\pp$ of $K$ we say the cover $L$ has {\em full local degree} at $\pp$ if
$[L:M]_\pp=[L:M]$ for finite $\pp$,
or $[L:M]_\pp=\gcd(2,[L:M])$ for real $\pp$,
or if $\pp$ is complex.
For a set $S$ of primes of $K$ we say $L$ has {\em full local degree in $S$} if $L$ has full local degree at each $\pp\in S$.
If $\pp$ is not isolated then the condition $d_\pp(m)|[L:M]_\pp$ is equivalent to saying that $L$ has full local degree at $\pp$.

\begin{lem}\label{lem:red}
Let $m'|m$.
Suppose there is a finite set $S_0$ of non-isolated primes of $K$
such that any $m$-cover of $M/K$ with full local degree in $S_0$ contains an $m'$-cover.
Then \eqref{eq:Bm} implies \Bm{m'}.
\end{lem}
\begin{proof}
For a given finite set $S$ let $L$ be an $m$-cover as predicted by
\eqref{eq:Bm} for $S\cup S_0$.
Since the primes in $S_0$ are non-isolated, $L$ has full local degree
in $S_0$.
By assumption, $L$ contains an $m'$-cover $L'$ of $M/K$.
For all primes $\pp$ of $K$,
$d_\pp(m)|[L:M]_\pp$ implies $d_\pp(m')|[L':M]_\pp$.
Hence, $L'$ satisfies the conditions required by ($\textrm{B}_{m'}$).
\end{proof}

\subsection{Reduction to prime powers}\label{sec:existencebounds}

Our next step is to reduce \eqref{eq:Bm} from arbitrary $m\in\N$ to prime-powers.
Except for the part concerning the characteristic of $K$ (Lemma \ref{wild.lem} below),
this reduction is identical to the corresponding one in \cite{hanke-sonn:location}.

Let $m=\prod p^{n_p}$ be the prime factorization.
By taking field composita of covers, if \Bm{p^{n_p}} holds for all $p|m$ then \eqref{eq:Bm} holds.
We show
\begin{prop}
  \label{prop:ppower}
\eqref{eq:Bm} holds if and only if \Bm{p^{n_p}} holds for all $p|m$.
\end{prop}
We first treat the wild case separately:
\begin{lem}\label{wild.lem} If $p=\charak K$ then \Bm{p^n} holds for all $n\in\N$.
\end{lem}
\begin{proof} Let $n\in\mathbb{N}$ and $S$ a finite set of primes of $K$.
For every $\fp\in S$ there is a cyclic $p^n$-extension $L'(\fp)/K_\fp$ which is disjoint from $M_\fp$
(see~\cite[Satz~10.4]{koch:perw} \footnote{Koch's book was translated to English. However, Theorem 10.4 in the English version contains a typo: ``finitely generated" should be replaced by ``on countably many generators".}).
By the Grunwald-Wang theorem, there is a cyclic $p^n$-extension $L'/K$ whose completion at $\fp$ is $L'_\fp=L'(\fp)$ for all $\fp\in S$. Let $L:=L'M$. Since $L'_\fp\cap M_\fp=K_\fp$, one has $[L:M]_\fp=p^n$ for all $\fp\in S$. Thus, $L$ is a $p^n$-cover of $M/K$ with full local degree in~$S$.
\end{proof}
It remains to show that \eqref{eq:Bm} implies \Bm{p^{n_p}} for all $p|m$
with $p\neq\charak K$.
\begin{proof}[Proof of Prop.\ \ref{prop:ppower}]
Let $p|m$ with $p\neq\charak K$.
By \cite[\S7, p.\ 325, Corollary]{hanke-sonn:location},
there are infinitely many primes $\pp$ of $K$ such that any $m$-cover of $M/K$ with full local degree at $\pp$ contains a $p^n$-cover of $M/K$.
(Note that the assumption $M/K$ cyclic is never used in the proof of \cite[\S7, p.\ 325, Corollary]{hanke-sonn:location}.)
Since there are only finitely many isolated primes, we can choose such
a $\pp$ that is non-isolated.
The proof is completed by setting $S_0:=\set{\pp}$ in Lemma \ref{lem:red}.
\end{proof}

\subsection{An invariant subgroup}\label{sec:inv}

We are now in the position to complete the proof of Theorem \ref{thm:main}.
As outlined in \S\ref{sec:thm},
and using Propositions \ref{prop:fld-reduction} and \ref{prop:ppower},
it remains to prove Proposition \ref{prop:star} below.
Let $p\in\Ps$ be fixed.
\begin{prop}\label{prop:star}
For any $n\in\N$, \Bm{p^n} implies \Bm{p^{n-1}}.
\end{prop}
Indeed,
the bound $b_p(\chi)$ of Theorem \ref{thm:main} is the maximal $n$ for which \Bm{p^n} holds for $M=K(\chi)$ if such a maximum exists,
and $b_p(\chi)=\infty$ otherwise.
For more details on the description of $b_p(\chi)$ see Corollary \ref{cor:bpn} below.

For any cover $L$ of $M/K$ we consider the group extension
\begin{equation}\label{eq:grpext}
 1\ra\Gal(L/M)\ra \Gal(L/K)\stackrel{}{\ra} \Gal(M/K)\ra 1.
\end{equation}
We will analyze several kinds of constraints that are imposed on \eqref{eq:grpext}
by the condition that $L$ has full local degree in $S_0$,
for certain chosen sets $S_0$.
More precisely,
after showing that the kernel $A:=\Gal(L/M)$ can be assumed to be abelian,
we focus on constraints regarding the conjugation action of $B:=\Gal(M/K)$ on $A$.
The analysis of this action is the main ingredient in the proofs of both
Proposition \ref{prop:star} and Theorem \ref{thm:finitness} below.

In view of Lemma \ref{wild.lem} we assume from now on that $p\neq\charak K$.
Fix $p^n$ and set $T:=M\cap K(\mu_{p^\infty})$.
For the proof of Proposition \ref{prop:star},
it will suffice to analyze the action of $\Gal(M/T)$ on $A$.

\begin{lem}\label{lem:S0}
  There exists a finite set $S_0$ of non-isolated primes of $K$ such
  that for any $p^n$-cover $L$ of $M/K$ with full local degree in $S_0$:
  the kernel $A$ is abelian,
  the group $\Gal(M/T)$ acts trivially on $A$,
  and $A$ has rank at most $2$.
\end{lem}
\begin{proof}
  At first fix a $\sigma\in\Gal(M/T)$.
  By the Chebotarev density theorem,
  the Galois extension $M(\mu_{p^n})/K$ has infinitely many unramified
  finite primes $\PP$ of $M(\mu_{p^n})$ whose Frobenius element
  restricts to the identity on $K(\mu_{p^n})$ and to $\sigma$ on $M$.
  Of those infinitely many $\PP$ choose one such that
  $\pp:=\PP\cap K$ is non-$p$-isolated in $M/K$ and $p$ does not
  divide the norm $N(\pp)$.
  (There are only finitely many $\PP$ that do not satisfy this, since $p\neq\charak K$.)
  Since $p\ndivides N(\pp)$ and $\mu_{p^n}\subset K_{\pp}$,
  we have $N(\pp)\equiv 1\pmod{p^n}$.

  Assume a $p^n$-cover $L$ of $M/K$ has full local degree at the
  chosen $\pp$.
  Let $\PP$ be a prime of $L$ dividing $\pp$.
  The decomposition group $\Gal(L_\PP/K_\pp)$ then equals $\Gal(L/M^\sigma)$ where $M^\sigma$ is the fixed field of $\sigma$.
  We will show that  $\Gal(L_\PP/K_\pp)$ is abelian, thus $\sigma$ acts trivially on $\Gal(L/K)$:
  Since $\charak K\ndivides N(\pp)$, $L_\PP/K_\pp$ is tame.
  Therefore, $\Gal(L_\PP/K_\pp)$ is a metacyclic group,
generated by the inertia group $I_\PP$ and the Frobenius element,
and the Frobenius acts on $I_\PP$
by raising each element to the power $N(\pp)$.
  Since $\pp$ is unramified in $M$,
  $|I_\PP|$ divides $p^n$.
  Hence $\Gal(L_\PP/K_\pp)$ is abelian, because $N(\pp)\equiv 1\pmod{p^n}$.

  Now let $\Sigma$ be a set of generators of $\Gal(M/T)$, or $\Sigma=\set{1}$ if $M=T$.
  For each $\sigma\in\Sigma$ choose a prime $\pp_\sigma$ as described
  in the first paragraph of the proof.
  Then $S_0:=\sett{\pp_\sigma}{\pp\in\Sigma}$ has the desired property.
\end{proof}
By Lemma \ref{lem:red},
Proposition \ref{prop:star} is completed once we show
\begin{prop}\label{prop:central-p}
Let $S_0$ be as in Lemma \ref{lem:S0}.
Any $p^n$-cover of $M/K$ with full local degree in $S_0$
contains a $p^{n-1}$-cover of $M/K$.
\end{prop}
\begin{proof}
We have assumed $p\neq\charak K$.
Let $B=\Gal(M/K)$.
Let $L$ be a $p^n$-cover of $M/K$ with full local degree in $S_0$.
By Lemma \ref{lem:S0}, the kernel $A$ is abelian.
The subgroup $A[p]$ of $p$-torsion elements is a characteristic subgroup and hence invariant under the action of $B$ (we say {\em $B$-invariant}).
It suffices to find a $B$-invariant subgroup $A_0\leq A[p]$ of order $p$;
then the fixed field $L^{A_0}$ is the desired $p^{n-1}$-cover.
If $A$ is cyclic then $A[p]$ itself is such a subgroup,
hence assume $A$ non-cyclic for the rest of the proof.

Recall that $A[p]$ is an $\F_p$-vector space and that any action of some group $H$ on $A[p]$ is a representation of $H$ over $\F_p$.
In this sense, the $H$-invariant subgroups of order $p$ are the $H$-invariant subspaces of dimension $1$.

By Lemma \ref{lem:S0}, $B$ acts on $A$ through $\Gal(T/K)$.
Let $\Gal(T/K)=P\oplus C$ with $P$ the $p$-part and $|C|$ prime-to-$p$.
Then $|C|$ divides $p-1$.
Since $|C|$ is prime to $p$, any representation of $C$ over $\mathbb{F}_p$ is semisimple and hence decomposes into a product of irreducible representations.
Since $\mathbb{F}_p$ contains the $|C|$-th roots of unity, the irreducible representations of $C$ are of dimension $1$.
Thus,  there is a $C$-invariant subgroup $A_0\leq A[p]$ of order $p$.

By Lemma \ref{lem:S0}, $A$ has rank $2$, i.e.\ $A[p]\cong C_p\times C_p$.
This group has exactly $p+1$ order $p$ subgroups, say $A_0,\ldots,A_p$,
which are permuted by the action of $B$.
Thus, we have an induced action of $B$ on the set of indices $\set{0,\ldots,p}$.

We know $A_0$ is $C$-invariant.
If $A_0$ is $P$-invariant then $A_0$ is $B$-invariant and we are done.
Assume $A_0$ is not $P$-invariant.
Since $P$ is a $p$-group there are two $P$-orbits on $\set{0,\ldots,p}$, say $\set{0,\ldots,p-1}$ and $\set{p}$.
Since $P$ and $C$ commute and $A_0$ is $C$-invariant, each of $A_0,\ldots,A_{p-1}$ is also $C$-invariant.
Hence the remaining subgroup $A_p$ is $C$-invariant and $P$-invariant and we are done.
\end{proof}

\subsection{Summary}

Let $\chi\in\Hom(G_K,\Delta/\Gamma)$ and $M=K(\chi)$.
\begin{cor}\label{cor:bpn}
The bound $b_p(\chi)$ is the maximal $n$ such that for every finite set $S$ of primes of $K$
there is a $p^n$-cover $L$ of $M/K$ with the following properties:
\begin{enumerate}
\item\label{it:dm} $d_\pp(p^n)\divides [L:M]_{\fp}$ for any $\pp\in S$,
\item\label{it:A} $A=\Gal(L/M)$ is abelian of rank at most $2$,
\item\label{it:act} $\Gal(M/T)$ acts trivially on $A$.
\end{enumerate}
If no maximal $n$ exists then $b_p(\chi)=\infty$.
\end{cor}
\begin{proof}
Let $S_0$ be a finite set of primes of $K$ that are non-isolated in $M/K$.
Suppose that any $p^n$-cover of $L$ with full local degree in $S_0$ has a certain property.
Then this property can be added to the condition \Bm{p^n} without
changing the truth value of \Bm{p^n}, because the set $S$ in \Bm{p^n}
can be enlarged by $S_0$.
By Lemma \ref{lem:S0}, this argument applies to the properties (ii) and (iii).
Hence, the Corollary is a consequence of Proposition \ref{prop:star}.
\end{proof}
\begin{rem}
Regarding condition (i) in Corollary \ref{cor:bpn},
if $\pp$ is not isolated then $d_\pp(p^n)|[L:M]_\pp$ is equivalent to saying that $L$ has full local degree at $\pp$.

Regarding condition (iii) in Corollary \ref{cor:bpn},
if $M$ and $K(\mu_{p^\infty})$ are disjoint over $K$ then (iii)
is equivalent to saying that the group extension \eqref{eq:grpext} is central.
\end{rem}

\section{Finiteness of bounds}\label{sec:finite}

\subsection{The exponents of kernels}
Suppose we are in the setup described in \S\ref{setup.sec}. 
Let $p\in\Ps$ be fixed and different from $\charak K$. Denote by $p^s$ the number of $p$-power roots of unity in $M$ and  by $r$ the maximal number for which $\mu_{2^r}\subseteq M(\sqrt{-1})$.
As before set $T:= K(\mu_{p^\infty})\cap M$.

In \S\ref{sec:exist} we showed that if a cover $L$ of $M/K$ has full local degree in $S_0$
then the choice of $S_0$ can put constraints on the action of $\Gal(M/T)$ on $A=\Gal(L/M)$ to the extent that this action is even trivial.
This was sufficient to prove the existence of the bounds.
Now, in order to prove the finiteness of the bounds, we analyze constraints on the action of the entire group $B=\Gal(M/K)$ on $A$.
The set $S_0$ for this purpose will be constructed from the families $Q_\sigma$ that we define next for each $\sigma\in B$.

Fix an element $\sigma\in B$ and let $f_\sigma$ be the order of the restriction $\sigma_{|T}$ of $\sigma$ to $T$.
We define $Q_\sigma$ to be the set of all primes $\fp$  of $K$, unramified in $M$, whose Frobenius automorphism in $M/K$ is $\sigma$,  and such that the norm $N(\fp)$ is prime to $p$ and is of order strictly greater than $f_\sigma$ as an element of $(\Z/p^{s+1}\Z)^*$ (resp.\ mod $(\Z/2^{r+2}\Z)^*$ if $p=2$).

\begin{lem}
For every $\sigma\in B$, the set $Q_\sigma$ is infinite.
\end{lem}
\begin{proof}
Assume without loss of generality that $\sqrt{-1}\in M$. Otherwise,  repeat the proof for a lift $\tau \in \Gal(M(\sqrt{-1})/K)$ of $\sigma$ to deduce that $Q_\tau$ is infinite. Since $Q_\tau\subseteq Q_\sigma$ and $f_\tau\geq f_\sigma$ the assertion for $\sigma$ follows. Note that under this assumption we have $r=s$, so we will use only $s$ in the rest of the proof.
Set $T':=K(\mu_{p^{s+1}})$ (resp.\ $T':=K(\mu_{2^{s+2}})$ if $p=2$) and note that $T'\cap M=T$.

We first claim that $\Gal(T'/K)$ contains an element $\sigma'$ of order greater than $f_\sigma$ whose restriction to $T$ is  $\sigma_{|T}$.
The group $\Gal(T/T^{\sigma})$ is naturally identified with a subgroup $U$ of $(\Z/p^s\Z)^*$,
and $\Gal(T'/T^{\sigma})$ is identified with the full preimage of $U$ under the natural projection
$\pi:(\Z/p^{s+1}\Z)^*\ra(\Z/p^s\Z)^*$ (resp.\ $\pi:(\Z/2^{s+2}\Z)^*\ra(\Z/2^s\Z)^*$ if $p=2$).
The claim follows since each element of $(\Z/p^s\Z)^*$ has a preimage under $\pi$ of a greater order.

Since $\sigma$ and $\sigma'$ agree on  $T$, Chebotarev's density theorem 
implies that there are infinitely many primes $\fp$ of $K$, with $p\ndivides N(\fp)$, whose Frobenius automorphism is $\sigma'$ in $T'/K$   and is $\sigma$ in $M/K$. 
Such primes are in $Q_\sigma$ since the order of the norm of $\fp$ as an element in $(\Z/p^{s+1}\Z)^*$ (resp.\ in $(\Z/2^{s+2}\Z)^*$)  is the same as  the order of their Frobenius automorphism in $K(\mu_{p^{s+1}})/K$ (resp.\ in $K(\mu_{2^{s+2}})/K$).
\end{proof}
For a prime $\fp$ of $K$ denote by $e_\fp(L/K)$ the ramification index of $\fp$ in a Galois extension $L/K$.
\begin{lem}\label{inertia.lem}
Let $S_0$ be as in Lemma~\ref{lem:S0} and let $\sigma\in B$.
Suppose that a $p^n$-cover $L$ of $M/K$ has full local degree in $S_0$ and at $\fp\in Q_\sigma$. 
Then $e_\fp(L/K)|p^s$ if $p$ is odd and $e_\fp(L/K)|2^{r+1}$ if $p=2$.
\end{lem}
\begin{proof} By Lemma \ref{lem:S0}, the kernel $A$ is abelian and the action of $B$ on $A$  factors through the action of $\Gal(M/T)$. Thus, $\sigma$ acts on the inertia group $I\subseteq A$ of $\fp$ in $L/K$ as an automorphism of order at most $f_\sigma$.

Assume on the contrary there is an element $a\in I$ of order $p^{s+1}$ (resp.\ $2^{r+2}$ if $p=2$).
Since $\fp$ is tamely ramified in $L$, $\sigma$ acts on $I$ by raising each element to the power  $N(\fp)$ and hence defines an automorphism  of order greater than $f_\sigma$ on $\langle a\rangle$, contradiction.
 \end{proof}

\begin{prop}\label{finiteness.prop}
  Assume the $p$-Sylow subgroup of $B$ is non-cyclic.
  There exists a finite set $S_0$ of non-isolated primes of $K$ such that for any $p^n$-cover $L$ of $M/K$ with full local degree in $S_0$: $\Gal(L/M)$ is abelian of rank at most $2$ and exponent at most $p^s$ (resp.\ $2^{r+2}$ if $p=2$).
\end{prop}

  If the $p$-Sylow subgroup of $B$ is non-cyclic then
Proposition \ref{finiteness.prop} allows us to improve on Corollary \ref{cor:bpn} by adding the following property to the list:
\begin{enumerate}
\item[(iv)] $\exp A|p^s$ if $p$ is odd and $\exp A|2^{r+2}$ if $p=2$.
\end{enumerate}
In particular, since $A$ has rank at most $2$,
$b_p(\chi)\leq 2s$ if $p$ is odd and $b_p(\chi)\leq 2(r+2)$ if $p=2$.
This proves Theorem \ref{thm:finitness}.

Proposition  \ref{finiteness.prop} relies on the following group theoretical proposition,
whose proof is the subject of subsection \S\ref{sec:central} below.
\begin{prop}\label{rank-element.lem} Let  
\begin{equation}\label{equ:prop-ext} 1\ra A\ra G \stackrel{\pi}{\ra} B \ra 1\end{equation}
 be an extension of non-trivial abelian $p$-groups $A,B$.
If $B$ is non-cyclic and $\pi^{-1}\langle x\rangle$ is cyclic for all $x\in B$ then $|A|=2$.
\end{prop}
\begin{proof}[Proof of Proposition \ref{finiteness.prop}]
Replacing $K$ with the fixed field of the $p$-Sylow subgroup of
  $B$ we can assume wlog.\ that $B$ is a $p$-group.
By assumption $B$ is non-cyclic (abelian).

Choose $S_0$ to be the set from Lemma \ref{lem:S0} joined with one non-isolated prime $\fp_\sigma\in Q_\sigma$ for each $\sigma\in B$.
Suppose $L$ is a $p^n$-cover of $M/K$ with full local degree in $S_0$.

By Lemma \ref{lem:S0}, 
$A$ is abelian of
rank at most $2$. Since the $p^s$-torsion subgroup $A[p^s]$ is a
characteristic subgroup of $A$, it is a normal subgroup of $\Gal(L/K)$
and hence the fixed field $L_0:=L^{A[p^s]}$ is Galois over $K$.
If $p=2$ then we consider $L_0:=L^{A[2^{r+1}]}$ instead, which is also Galois over $K$.
To prove our claim it suffices to show that $L_0=M$  (resp.\ $[L_0:M]\leq 2$ if $p=2$).

Fix an element $\sigma \in B$ and let $\frak{P}_\sigma$ be a prime of $L$ which divides $\fp_\sigma$.
Let $I_\sigma\subseteq A$ be the inertia group of $\frak{P}_\sigma$ in $L/K$.
By Lemma \ref{inertia.lem}, $|I_{\sigma}|\leq p^s$ (resp.\ $|I_{\sigma}|\leq 2^{r+1}$)
and hence $I_\sigma\subseteq A[p^s]$ (resp.\ $I_\sigma\subseteq A[2^{r+1}]$).
Thus, $L_0/K$ is unramified at $\fp_\sigma$ and hence $\frak{P}_\sigma\cap L_0$ has a cyclic decomposition group in $L_0/K$. Since $L_0$ has full local degree at $\fp$, the decomposition group of $\frak{P}_\sigma\cap L_0$ in $L_0/K$ is $\Gal(L_0/M^\sigma)$ and hence $L_0/M^\sigma$ is a cyclic extension.

Since $L_0/M^\sigma$ is cyclic  for all $\sigma\in B$ and since $B$ is non-cyclic,
Proposition \ref{rank-element.lem} applied to the group extension
\[ 1\to\Gal(L_0/M)\to\Gal(L_0/K)\to\Gal(M/K)\to 1 \]
shows $L_0=M$ (resp.\ $[L_0:M]\leq 2$ if $p=2$),
proving the claim.
\end{proof}

\subsection{Central group extensions}\label{sec:central}
The remaining ingredient is a proof of Proposition \ref{rank-element.lem}.
We begin with elementary properties of commutators in
a central group extension
\begin{equation}\label{equ:extension} 1\ra A\ra G \ra B \ra 1\end{equation}
of abelian groups $A,B$.
Let $s:B\to G$ be a section of $G\to B$ (not necessarily a homomorphism).
\begin{lem}\label{lem:bimult}
Commutators in $G$ are bimultiplicative.
The map
\[\beta:B\times B\to A, (x,y)\mapsto [s(x),s(y)]\]
does not depend on the choice of $s$ and is bimultiplicative.
\end{lem}
\begin{proof}
Since $[G,G]\subseteq A\subseteq Z(G)$ we have
\[[ab,x]=abxb^{-1}a^{-1}x^{-1}=a(bxb^{-1}x^{-1})xa^{-1}x^{-1}=[a,x][b,x].\]
Similarly one checks $[x,ab]=[x,a][x,b]$, i.e.\ commutators are bimultiplicative.
The statement about $\beta$ follows from this.
\end{proof}
We next look at the meaning of the condition that $\pi^{-1}\gen{x}$ is cyclic for $x\in B$.
For $x=1$ it means $A$ is cyclic, and for $x\neq 1$ one has:
\begin{lem}\label{lem:pi-inv}
Assume $A$ is cyclic.
For $x\in B,x\neq 1$,
$\pi^{-1}\gen{x}$ is cyclic if and only if $A$ is trivial or generated by $s(x)^{\ord{x}}$.
\end{lem}
In order to prove Proposition \ref{rank-element.lem} we now assume
$A,B$ are non-trivial $p$-groups and $A$ is cyclic.
The map
\[\gamma:B[p]\to A/A^p, x\mapsto s(x)^p\]
is obviously independent of the choice of $s$.
\begin{lem}\label{lem:gamma-hom}
Assume $A,B$ are non-trivial $p$-groups and $A$ is cyclic.
  \begin{enumerate}
  \item For $x\in B[p], x\neq 1$, $\pi^{-1}\gen{x}$ is cyclic if and only if $\gamma(x)\neq 1$.
  \item If $p$ is odd then $\gamma$ is a homomorphism.
  \item If $p=2$ then $\gamma$ is a homomorphism if and only if $\beta(x,y)\in A^2$ for all $x,y\in B[2]$.
  \end{enumerate}
\end{lem}
\begin{proof}
(i) is Lemma \ref{lem:pi-inv}.
Since $[G,G]\subseteq A\subseteq Z(G)$, we have $(s(x)s(y))^p=s(x)^ps(y)^p\beta(x,y)^{\frac{p(p-1)}{2}}$ for all $x,y\in B$.
In particular, for all $x,y\in B[p]$:
\[\gamma(xy)=s(xy)^pA^p=(s(x)s(y))^pA^p=\gamma(x)\gamma(y)\beta(x,y)^{\frac{p(p-1)}{2}}.\]
Thus $\gamma$ is a homomorphism if and only if $\beta(x,y)^{\frac{p(p-1)}{2}}\in A^p$.
If $p=2$ then $\frac{p(p-1)}{2}=1$, proving (iii).
For $p$ odd we use Lemma \ref{lem:bimult} to see that $\beta(x,y)\in A[p]$ for all $x,y\in B[p]$,
so the term $\beta(x,y)^{\frac{p(p-1)}{2}}$ vanishes.
\end{proof}
\begin{proof}[Proof of Proposition \ref{rank-element.lem}]
Let $A,B$ be non-trivial abelian $p$-groups, $A$ cyclic and $B$ non-cyclic.
By hypothesis and (i) of Lemma \ref{lem:gamma-hom}, $\gamma(x)\neq 1$ for all $x\in B[p],x\neq 1$.
Therefore, if $\gamma$ is a homomorphism then it is injective, in contradiction to $A$ being cyclic and $B$ non-cyclic.
Hence, $\gamma$ is not a homomorphism.
By Lemma (ii) and (iii) of \ref{lem:gamma-hom}, we have $p=2$ and an element in $A[2]\setminus A^2$.
This implies $|A|=2$.
\end{proof}

\section{Examples}\label{examples.sec}

Suppose we are in the setup described in \S\ref{setup.sec}.
For $p\in\Ps$, let $p^{s_p(M)}$ denote the number of $p$-power roots of unity in $M$.

If $\alpha+\chi$ has index equal to $|\chi|$ then the division algebra
contained in $\alpha+\chi$ is a crossed product,
because $\alpha$ is split by $K(\chi)$.
Therefore, noncrossed products of index $p^2$ are possible only if $|\chi|=p$,
in particular only if $\chi$ is cyclic.

Suppose $\chi$ is non-cyclic with $|\chi|=p^2$.
If $b_p(\chi)=0$ then the fiber over $\chi$ contains infinitely many noncrossed products of index $p|\chi|$.
By Theorem \ref{thm:finitness}, this happens, e.g., whenever $s_p(M)=0$.
The present section gives examples of bicyclic $\chi$ with $|\chi|=p^2$ and $b_p(\chi)=0$ but $s_p(M)\geq 1$.
We point out that such a phenomenon is in contrast to the case of cyclic $\chi$ for which one always has $b_p(\chi)\geq s_p(M)$
(see \cite{hanke-sonn:location}).

For $p=2$ an example as described above was given over $K=\mQ$ in \cite{hanke:expl-ex}
and over $K=\F_q(t)$ for all $q\equiv 3\pmod{8}$ in \cite[Example 2.8]{Coyette}.
These turn out to be special cases of our Examples \ref{ex:2} and \ref{ex:Fq} below.

We start with $K=\Q$ and $p=2$:
\begin{ex}\label{ex:2} Let $q,l$ be odd primes such that  $q\equiv 3\pmod{4}, q\not\equiv -l\pmod{8}$
and $q$ a non-square modulo $l$.
Note that for any odd prime $l$ a suitable $q$ can be chosen using Dirichlet's theorem.
\footnote{Using the reciprocity law, it is also possible to choose a suitable $l$ for any prime $q\equiv 3\pmod{4}$.}
%
Let $\mQ(\chi)=\mQ(\sqrt{q},\sqrt{-l})$
Corollary \ref{cor:bpn} and the following claim show $b_p(\chi)=0$.
\end{ex}
\begin{claim}
The extension $M/\Q$ has no isolated primes and there is no $2$-cover $L$ of $M/K$ with local degree $[L:M]_{l}=2$. \end{claim}
\begin{proof}
Set $K_1:=\Q(\sqrt{q}), K_2:=\Q(\sqrt{-l})$ and $M:=K(\chi)$.

We first check that $M/\mQ$ has no isolated primes.
The prime $l$ ramifies in $K_2$ and is inert in $K_1$, so $[M:\Q]_l=4$.
{\em Case} $l\equiv 3\pmod{4}$:
By reciprocity, $l$ is a square modulo $q$, hence $-l$ is a non-square
modulo $l$.
The prime $q$ thus ramifies in $K_q$ and is inert in $K_2$, so $[M:\Q]_q=4$.
{\em Case} $l\equiv 1\pmod{4}$:
Since $q\not\equiv -l\pmod{8}$, we have $\Q_2(\sqrt{q})\neq\Q_2(\sqrt{-l})$, so $[M:\Q]_2=4$.
In any case, $M/\Q$ has no isolated prime.

Now assume $L$ is a $2$-cover of $M/\Q$ with full local degree at $l$.
Since $K_1$ is real and $M$ is not, $M/K_1$ does not have a cyclic $2$-cover.
Since $q\equiv 3\pmod{4}$, $-1$ is not a square in $\Q_q$.
This implies that $\Q_q$ does not have any totally ramified
degree $4$ extension,
so that any ramified quadratic extension of $\Q_q$ cannot
have a cyclic $2$-cover.
Thus, globally, $K_1/\Q$ does not have a cyclic $2$-cover.
The inertia field of $l$ in $L/\Q$ contains $K_1$ and is cyclic
over $\Q$, thus is equal to $K_1$.
This is a contradiction because $L$ is then a cyclic $2$-cover of
$M/K_1$.
\end{proof}

\begin{rem*}
(i) Suppose $\mQ(\chi)/\Q$ is as in Example \ref{ex:2}.
    Consider $\alpha\in\Br(\Q)$ such that $\ind\alpha=8$ and $\ind\alpha^{\Q(\chi)}=2$.
    Since $l$ is not $2$-isolated in $\mQ(\chi)$, we can find such an
    $\alpha$ with $\ind_l\alpha=8$.
    Since $\mQ(\chi)/\mQ$ does not have a $2$-cover $L$ with $[L:\mQ(\chi)]_l=2$,
    no $2$-cover of $M/K$ splits $\alpha$.
    Hence, the underlying division algebra of $\alpha+\chi$ is a
    noncrossed product of index $8$.\\
(ii) In Example \ref{ex:2} we can choose
    $(l,q)=(3,11),(5,7),(7,3),$ 
    etc.
    The example in \cite{hanke:expl-ex} is the case $l=7$ and $q=3$.\\
\end{rem*}
We now turn to arbitrary global fields $K$ and a prime $p\neq\charak K$.
Example \ref{ex:2} does not generalize immediately because its proof uses a real prime.
The following argument uses a third finite prime instead of a real prime:
\begin{prop}\label{ex:3}
Let $K$ be a global field and let $p\in\Ps$ with $p\neq\charak K$.
Assume $s:=s_p(K)>0$.
Let $\pp$ be any prime of $K$ with $p\ndivides N(\pp)$.
There exists a bicyclic extension $M/K$ with group
  $C_{p^s}\times C_{p^s}$ and without isolated primes such that no
  $p$-cover $L$ of $M/K$ has $[L:M]_{\pp}=p$.
\end{prop}
\begin{proof}
By Chebotarev's density theorem, there are primes $\fq_1,\fq_2$ of $K$ with $N(\fq_i)\equiv 1\pmod{p^s}$
but $N(\fq_i)\not\equiv 1\pmod{p^{s+1}}$.
By the Grunwald-Wang theorem there are cyclic extensions $K_i/K$ of degree $[K_i:K]=p^s$, such that in $K_1$: $\pp$ is
inert, $\fq_1$ is totally ramified, $\fq_2$ splits completely,
and in $K_2$:  $\fq_1$ is inert and $\pp, \fq_2$ are totally ramified.
Since $\pp$ and $\fq_1$ both have full local degree in $M$, $M/K$ has
no isolated primes.

Since $N(\fq_1)\not\equiv 1$ (mod $p^{s+1}$) and $\fq_1$ is totally ramified in $K_1/K$, $K_1/K$ does not have a cyclic $p$-cover.
Similarly, considering $\fq_2$, $M/K_1$ does not have a cyclic $p$-cover.

Assume on the contrary there is a $p$-cover $L$ of $M/K$ with $[L:M]_{\pp}=p$.  Since the inertia field of $\pp$ contains $K_1$ and is cyclic over $K$, it equals $K_1$. This shows that $L$ is a cyclic $p$-cover of $M/K_1$, contradiction.
\end{proof}
\begin{ex}\label{ex:Fq}
Let $p\in\Ps$ and $K=\F_q(t)$ for $q\equiv 1\pmod{p}$, so that $s:=s_p(K)>0$.
Assume $a\not\in(K^\times)^p$.
Let $K(\chi)=K(\sqrt[p^s]{t},\sqrt[p^s]{(t-1)(t-a)})$.
By the following claim, the proof of Proposition \ref{ex:3} applies to $M=K(\chi)$ and
the primes $\pp=(t-a),\fq_1=(t),\fq_2=(t-1)$.
Therefore, $b_p(\chi)=0$.
\end{ex}
For $q\equiv 3\pmod{4}$ and $p=2$, Example \ref{ex:Fq} is identical with \cite[Example 2.8]{Coyette}.
\begin{claim} Let $K_1=K(\sqrt[p^s]{t}), K_2=K(\sqrt[p^s]{(t-1)(t-a)})$. Then $(t-a)$ is inert in $K_1$ and totally ramified in $K_2$, $(t)$ is totally ramified in $K_1$ and inert in $K_2$, and $(t-1)$ splits completely in $K_1$ and is totally ramified in $K_2$.  
\end{claim}
\begin{proof}

In $K_1$ we have:
$t\equiv a\pmod{t-a}$ is not a $p$-th power
and $t\equiv 1\pmod{t-1}$ is a $p^s$-th power,
hence $(t-a)$ is inert, $(t-1)$ splits completely, and $(t)$ is totally ramified.

In $K_2$ we have:
$(t-1)(t-a)\equiv a\pmod{t}$ is not a $p$-th power,
hence $(t)$ is inert and $(t-1),(t-a)$ are totally ramified.
\end{proof}

\def\cprime{$'$}

\end{document}